\documentclass[10pt]{amsart}

\usepackage{graphicx,comment} 

\usepackage{amsthm,amsmath,amsfonts,amssymb}
\usepackage{enumitem,todonotes,url}
\usepackage{tikz,tikz-cd}
\usepackage[all]{xy}
\usepackage{xcolor}
\usepackage[normalem]{ulem}

\usepackage{dsfont}
\usetikzlibrary{matrix}
\usepackage{float}
\usepackage{hyperref}
\numberwithin{equation}{section}

\theoremstyle{plain}
\newtheorem{theorem}[equation]{Theorem}
\newtheorem{corollary}[equation]{Corollary}
\newtheorem{prop}[equation]{Proposition}
\newtheorem{lemma}[equation]{Lemma}

\newtheorem{eg}[equation]{Example}
\newtheorem{defn}[equation]{Definition}
\newtheorem{Quest}[equation]{Question}

\newcommand{\Img}{\operatorname{Im}}

\newcommand{\mc}[1]{\mathcal{#1}}

\newcommand{\mr}[1]{\mathrm{#1}}

\newcommand{\mit}[1]{\mathit{#1}}

\newcommand{\abs}[1]{\lvert #1 \rvert}

\newcommand{\bra}[1]{\langle #1 \rangle}

\newcommand{\td}[1]{\widetilde{#1}}

\newcommand{\ZZ}{\mathbb{Z}}
\newcommand{\RR}{\mathbb{R}}

\DeclareMathOperator{\Map}{Map}

\title[Quasimonophobic graphs and degree spectral sequences]
{Quasimonophobic graphs and degree spectral sequences in discrete cubical homology}
\author{Samira Sahar Jamil}
\address{Department of Mathematics, University of Notre Dame, Notre Dame, Indiana}

\author{Mark Behrens}
\address{Department of Mathematics, University of Notre Dame, Notre Dame, Indiana}

\date{\today}

\begin{document}
\begin{abstract}
We introduce the degree filtration on the discrete cubical chain complex of a graph, defined in terms of the maximal injective dimension of the facets of singular $n$-cubes, and study the degree spectral sequence which arises from this filtration.  This spectral sequence interpolates between the discrete cubical homology of a graph $H_n(G)$ and the injective homology $H_n^{inj}(G)$, a variant of the discrete cubical homology based on injective singular cubes. Building on the work of Greene and the first author, we introduce the combinatorial condition of quasimonophobicity on graphs, and show quasimonophobicity implies both the vanishing of the degree spectral sequence in certain bidegrees, and implies $H_n^{inj}(G)$ is isomorphic to the homology of the CW complex obtained by ``filling in'' subcubes of the graph.  These results are applied to compute $H_2(G_n^{sph})$ for the Greene sphere graphs $G^{sph}_n$.
\end{abstract}

\maketitle

\section{Introduction}
In \cite{barcelo2014_discrete_metric}, Barcelo, Capraro, and White introduced the discrete cubical homology groups\footnote{These are sometimes denoted $\mc{H}_n^{\mr{Cube}}(G)$ --- see  \cite{barcelo2019discrete}.}
$$ H_n(G) $$ 
of a graph $G$.  These homology groups capture the higher-dimensional structure beyond that seen in classical homologies (where one considers the homology of the $1$-dimensional CW complex associated to a graph, or the homology of a simplicial complex, such as the clique complex, associated to a graph). In contrast to these classical homologies, discrete cubical homology probes a graph using discrete cubes of all dimensions, yielding an interesting theory both from a theoretical perspective and for its applications.

Discrete cubical homology, together with its associated homotopy theory \cite{barcelo2001foundations}, raises interesting questions in higher category theory. For instance, computational considerations of discrete cubical homology led to the result that the homology of a cubical set is unchanged when a subcomplex consisting of connection cubes is quotiented out from the chain complex \cite{barcelo2021connections}. Similar studies motivated the comparison between homotopy $n$-types of graphs and those of cubical and simplicial sets \cite{carranzakapulkincubical}, \cite{kapulkin2024homotopy_n_types}, \cite{carranzakapulkinhomotopyhyp}. On the other hand, from an applied viewpoint, discrete cubical homology has also been shown to be more robust to noise than the homology of the flag complex of a graph, making it useful for persistent homology and topological data analysis \cite{kapulkin2025data_analysis_discrete_homology}.

Despite its conceptual appeal and broad applications, 
discrete cubical homology is shockingly difficult to compute in degrees greater than $1$, even for some very simple graphs.  The primary reason for this is that the hypotheses for excision in discrete cubical homology are much more constrained than in the topological or simplicial contexts \cite[Thm.~3.4]{barcelo2014_discrete_metric}.  While discrete cubical homology of a finite graph is in principle finitely computable, in practice, even with the implementation of efficient algorithms, actual machine computation of many simple graphs is limited to degrees less than or equal to 3 \cite{kapulkin2024efficient}.
This difficulty highlights the limited understanding of the nature of the higher-dimensional graph cells contributing to nontrivial homology classes.

At present, it is only known that higher-dimensional cells cannot arise in the absence of triangles and squares in the graph. Indeed, a key result of \cite{barcelo2021SIAM} shows that graphs containing no triangles or squares have trivial discrete cubical homology in dimensions $2$ and higher. This result is consistent with earlier work in discrete homotopy theory, where triangles and squares play the role of $2$-cells: the discrete fundamental group agrees with the fundamental group of the CW complex obtained by attaching $2$-cells along all triangles and squares 
\cite{barcelo2001foundations}. 
It follows that graphs can be used to represent homotopy $1$-types.  Homotopy $n$-types for $n\geq 2$ continue to be rather mysterious objects of study \cite{kapulkin2024homotopy_n_types}, though Carranza and Kapulkin have recently shown the striking result that every homotopy $n$-type can be abstractly represented by a graph \cite{carranzakapulkinhomotopyhyp}.

These results naturally lead to the question of whether there exists an intrinsic notion of dimension for graphs that governs the vanishing of discrete cubical homology in higher degrees. A natural candidate is the \emph{cubical dimension} of a graph, defined as the maximal dimension of an injective singular cube it contains. It is unclear whether this notion alone is sufficient: in principle, non-cubical subgraphs could contribute to homology in degrees exceeding the cubical dimension, though we know of no actual examples of this phenomenon occurring.

In this paper, we will define a \emph{degree filtration} on the discrete cubical singular chains of a graph which arises from consideration of the maximal injective dimension of the facets of a singular cube. Associated to this filtration is a \emph{degree spectral sequence}
$$ E^1_{p,q} (G) =  H^{[p]}_{p+q}(G) \Rightarrow H_{p+q}(G) $$
where the homology groups $H^{[p]}_*(G)$ are the homology of the $p$th associated graded of our filtration and satisfy
$$ H^{[>n]}_n(G) = 0.$$
If $G$ has cubical dimension $d$, then
$$ H^{[>d]}_n(G) = 0. $$

We will define (Definition~\ref{defn:Hinj}) the \emph{injective homology} $H^{inj}_*(G)$ to be the edge of the $E_2$-page of this spectral sequence given by the homology of the chain complex
$$ H^{[0]}_0(G) \xleftarrow{d_1} H^{[1]}_1(G) \xleftarrow{d_1} \cdots. $$
There are maps
\begin{equation}\label{eq:edge}
 H_n(G) \twoheadrightarrow E^{\infty}_{n,0}(G) \hookrightarrow H^{inj}_n(G).
 \end{equation}

The injective homology promises to be much more computable than the discrete cubical homology, because it only involves injective singular cubes.  Indeed, we will show  (Theorem~\ref{thm:cell}) that if $CW^\Box(G)$ is the CW complex whose $n$-cells correspond to the subgraphs of $G$ isomorphic to $n$-cubes, then there is a surjection of chain complexes
\begin{equation}\label{eq:cell}
 C^{\emph{cell}}_n(CW^\Box(G)) \twoheadrightarrow H^{[n]}_n(G)
 \end{equation}
resulting in a map
\begin{equation}\label{eq:comparison}
 H_n(CW^\Box(G)) \to H^{inj}_n(G).
 \end{equation}
We then are motivated to ask:
\begin{Quest}\label{quest:comparison}
When are the map (\ref{eq:comparison}) and the maps (\ref{eq:edge}) isomorphisms, resulting in an isomorphism
$$ H_n(G) \cong H_n(CW^\Box(G))? $$
\end{Quest}

We will see that sometimes affirmative answers to Question~\ref{quest:comparison} derive from combinatorial conditions on the graph $G$.  For example, the results of \cite{barcelo2021SIAM} imply that 
\begin{equation*}
H_n(G) \cong H_n(CW^\Box(G))
\end{equation*}
for triangle-free square-free graphs.\footnote{For such graphs, $CW^\Box(G)$ is nothing other than the $1$-dimensional CW complex associated to the graph $G$.} Moreover, it follows from \cite[Thm.~3.4]{EnderKapulkin} (see also \cite[Prop.~5.12]{barcelo2001foundations}) that for triangle-free graphs $G$ we have\footnote{Ender and Kapulkin instead consider the $2$-dimensional CW complex $X_G$ obtained by filling in triangles and squares with $2$-cells, but for $G$ triangle-free, $X_G$ is just the $2$-skeleton of $CW^\Box(G)$.}
$$ H_1(G) \cong H_1(CW^\Box(G)). $$

The groups $H^{[0]}_{\ge 1}(G)$ are all zero for trivial reasons.
We will adapt the ``covering space'' technique used in \cite{barcelo2021SIAM} to prove that for $G$ triangle-free we
have (Theorem~\ref{thm_H_2_1=0})
$$ H^{[1]}_{\ge 2}(G) = 0. $$
It will follow from the degree spectral sequence (Corollary~\ref{cor_2_homology_no_3_cycles}) 
that if $G$ is a triangle-free graph of cubical dimension $2$, there are isomorphisms
$$
H_n(G) \cong 
\begin{cases}
H_n^{inj}(G), & n \le 2, \\
H_n^{[2]}, & n > 2.
\end{cases}
$$

The notion of monophobic cube neighborhoods, introduced in \cite{monophobic}, provides a different generalization of the triangle-free square-free condition considered by \cite{barcelo2021SIAM}, where one considers graphs in
which every $n$-cube has a monophobic neighborhood.

Indeed, a triangle-free and
square-free graph is precisely one in which every edge (viewed as a
$1$-cube) has a monophobic neighborhood (such graphs are referred to as $1$-monophobic, see Definition \ref{defn_monophobic}). In this paper, we introduce the class
of \emph{$n$-quasimonophobic} graphs (see Definition~\ref{defn_quasimonophobic}).  It will turn out that a graph is $1$-quasimonophobic if and only if it is triangle-free.

We will show (Theorem~\ref{thm:cell}) that if $G$ is $n$-quasimonophobic, then the map (\ref{eq:cell}) gives an isomorphism
\begin{equation}\label{eq:celliso}
 H^{[n]}_n(G) \cong C^{\emph{cell}}_n(CW^\Box(G)). \end{equation}
 We will then deduce (Theorem~\ref{thm:iso}) that for $G$ $n$-quasimonophobic, the map (\ref{eq:cell}) induces an injection
 $$ H_n(CW^\Box(G)) \hookrightarrow H^\mathit{inj}_n(G) $$
 and this map is an isomorphism if $G$ is $(n-1)$-quasimonophobic.

We then deduce our main result (Theorem~\ref{thm:main}) from the degree spectral sequence.  Suppose that $G$ is $(\le 2)$-quasimonophobic.  Then the map (\ref{eq:edge}) and map (\ref{eq:cell}) induce an isomorphism
$$ H_2(G) \cong H_2(CW^\Box(G)). $$
This appears to be the first systematic result computing discrete cubical $H_2$ with the exception of that of \cite{barcelo2021SIAM}, which establishes that $H_2(G) = 0$ for $G$ triangle-free and square-free. 

As an application, we compute the second discrete cubical homology of the Greene spheres $G^{sph}_n$. The graphs $G^{sph}_n$ ($n \ge 4$) are a family of graphs of cubical dimension $2$ which were constructed by Greene to model the $2$-sphere in the setting of discrete cubical homology.  Specifically, the graph $G^{sph}_4$ was introduced in \cite{barcelo2019discrete}, where the authors prove that $H_2(G^{sph}_4) \ne 0$, and machine computation is used in \cite{kapulkin2024efficient} to show that\footnote{In fact, \cite{kapulkin2024efficient} also compute $H_3(G^{sph}_4) = 0$.}
$$ H_2(G^{sph}_4) \cong \ZZ. $$
The Greene spheres turn out to be $n$-quasimonophobic for every $n$, and hence are amenable to Theorem~\ref{thm:main}.  We will deduce
(Theorem~\ref{thm_Greenes_sphere})
\begin{equation}\label{eq:Greene}
 H_2(G^{sph}_n) = \ZZ \qquad \text{for $n \ge 4$}.
 \end{equation}
Presumably, machine computations could be used to reproduce (\ref{eq:Greene}) for small $n$.

We also explain more generally how to  ``geometrically'' produce nontrivial classes 
in $H_n(G)$ when $G$ is $n$-quasimonophobic (Theorem~\ref{thm:geom}).

\textbf{Organization of the paper}  

In Section~\ref{sec_[prelim]} we set our basic definitions, and recall the definition of discrete cubical homology.  In Section~\ref{sec_[dss]} we define the degree filtration and the degree spectral sequence.  We show in Section~\ref{sec_H_n_1_equals_0} that triangle-free graphs $G$ have $H^{[1]}_{\ge 2}(G) = 0$.  In Section~\ref{sec_injective_cycles} we introduce the notion of $n$-quasimonophobicity, and show that for such graphs $G$ we have $H^{[n]}_n(G) \cong C^{\emph{cell}}_n(CW^\Box(G))$, and deduce our main results.  In Section~\ref{sec_[examples]} we apply our results to compute $H_2(G^{sph}_n)$.  We also give an example to show the necessity of quasimonophobicity for Theorem~\ref{thm:main} and contemplate possible generalization of Theorem~\ref{thm:main} for quasimonophobic graphs.  We end this section with a geometric construction of classes of $H_n(G)$, which we prove are non-trivial when $G$ is $n$-quasimonophobic.
\newpage
\textbf{Acknowledgments}

The first author thanks Curtis Greene and Georg Wille for helpful discussions; in particular, Wille brought Example~\ref{eg_K23} to her attention. The authors were also grateful for the comments Chris Kapulkin provided on an earlier draft.  This material is based in part upon work supported by the National Science Foundation under Grant No. DMS-1928930 and by the Alfred P. Sloan Foundation under grant G-2021-16778, while the first author was in residence at the Simons Laufer Mathematical Sciences Institute (formerly MSRI) in Berkeley, California, during the year 2024.  
This material is also based upon work supported by the National Science Foundation under Award No. DMS-2506564.  

\section{Preliminaries}\label{sec_[prelim]}

Throughout this document, the term “graph” denotes a simple undirected graph -- one without loops or multiple edges. We denote a graph as a pair $(V,E)$ of vertex set $V$ and edge set $E$, where each edge is a two-point subset of $V$. We sometimes denote the vertex and edge sets as $V(G)$ and $E(G)$, respectively. 

A \emph{graph map}
$$ f\colon G\to G'$$ 
from a graph $G=(V,E)$ to another graph $G'=(V',E')$ is a map $f\colon V\to V'$ between their vertex sets that maps every edge to an edge or a vertex. 
We shall say $f$ is an isomorphism if it has an inverse.  We will let 
$$ \Map(G,G') $$
denote the set of all graph maps from $G$ to $G'$.

The \emph{image} $f(G)$ or $\Img(f)$ of a graph map $f\colon  G\to G'$ is the subgraph of $G'$ with vertex set $f(V)$ and only those edges of $G'$ that are images of the edges in $G$. 

We shall write
$$ G \subseteq G' $$
to indicate that $G$ is a subgraph of $G'$, in the sense that $V(G) \subseteq V(G')$ and $E(G) \subseteq E(G')$.  A
subgraph $G \subseteq G'$ is \emph{induced} if for each pair of vertices $v,w \in V(G)$, an edge connects them in $G$ if and only if an edge connects them in $G'$.

The \emph{neighborhood graph} $N(v)$ of a vertex $v$ in a graph $G$ is the induced subgraph of $G$ containing $v$ and the vertices adjacent to $v$.
We say that a graph map $f: G \to G'$ is a local isomorphism if for all vertices $v\in V(G)$, $f$ restricts to an isomorphism 
$$ f\big\vert_{N(v)}: N(v) \to N(f(v)).  $$ 

For $n\geq 0$, we will let $I^n$ denote the \emph{discrete $n$-cube graph}, with the vertex set $\{0,1\}^n$, where two vertices are adjacent if they differ in exactly one coordinate. In particular, $I^0$ is the graph with a single vertex and no edges. 

For each $n \ge 0$, define the \emph{singular cubical $n$-chains on $G$} to be the free abelian group 
$$ Q_n(G) := \ZZ\Map(I^n,G). $$
generated by the \emph{singular $n$-cubes}
$$ \sigma \colon I^n \to G. $$
For $1 \le i \le n$, the \emph{front} and \emph{back $i$-faces} of such a singular $n$-cube $\sigma$
are denoted as $f_i^-\sigma$ and $f^+_i\sigma$, where $f_i^-\sigma$ is obtained by fixing the $i$-th coordinate equal
to $0$, while the back $i$-face $f^+_i\sigma$ is obtained by fixing the $i$-th coordinate equal
to $1$. These face operators satisfy the face relation: for $\sigma$ a singular $n$-cube and $1 \le i < j \le n$, we have 
\begin{align}\label{eq_face relation}
f_i^af_j^b\sigma=f^b_{j-1}f^a_i\sigma.
\end{align}

The \emph{boundary} operator $\partial_n \colon Q_n(G) \to Q_{n-1}(G)$ is defined on a singular $n$-cube $\sigma$ by
$$
\partial_n \sigma = \sum_{i=1}^n (-1)^i (f_i^- \sigma - f^+_i \sigma).
$$
We have $\partial_n \circ \partial_{n+1} = 0$, so that $(Q_*(G), \partial_*)$ forms a chain complex.  
However, as in the case with singular cubical chains for topological spaces, the homology of this chain complex is \emph{not} the desired discrete cubical homology of $G$ (e.g. if $G$ has 1 vertex, $H_n(Q_*(G))$ is nonzero in every degree).  We instead are required to normalize by modding out by degeneracies.

A singular $n$-cube $\sigma$ is said to be \emph{degenerate}, if there exists an index $i$ such that $f_i^- \sigma = f^+_i \sigma$. The boundary of a degenerate cube is a linear combination of degenerate cubes, so the degenerate singular cubes form a subcomplex
$$ D_*(G) \subseteq Q_*(G). $$
We define the \emph{normalized singular cubical chains} to be the quotient complex 
$$ C_*(G) := Q_*(G)/D_*(G). $$
The \emph{discrete cubical homology} is defined to be the homology of this quotient complex:
$$ H_n(G) := H_n(C_*(G)). $$

It should be noted that even for graphs with few vertices, the rank of these normalized singular cubical chain groups typically grows super-exponentially with respect to dimension, making the computation of discrete cubical homology highly challenging.

\section{The degree spectral sequence}\label{sec_[dss]}

The notion of degree of a singular cube, introduced in \cite{PhD_thesis_Samira} as the \emph{degree of injectivity}, quantifies the extent to which a singular cube is injective. 

\begin{defn}
The \emph{degree} of a singular $n$-cube 
$ \sigma\colon  I^n\to G $ 
is defined to be the dimension of the maximal injective facet of $\sigma$.
\end{defn}

Thus degree-$0$ singular cubes are constant, degree-$n$ singular $n$-cubes are injective, and degree-$(n-1)$ singular $n$-cubes are precisely the noninjective singular $n$-cubes with an injective face.

If $\sigma$ is a singular cube of degree $k$, then every face of $\sigma$ has degree at most $k$, \textit{i.e.}, the boundary operator $\partial$ does not, in general, preserve degree. This observation suggests filtering the cubical chain complex according to degree and studying the resulting associated graded complex, in which only the degree-preserving component of the boundary remains.
Let 
$$ C_n^k (G) \le C_n(G) $$ 
denote the subgroup generated by the non-degenerate singular $n$-cubes of degree $k$ or below. This results in a \emph{degree filtration}
$$
C^0_*(G) \le C^1_*(G) \le \cdots 
$$
on the chain complex $C_*(G)$.
We will denote the associated graded of this filtration by
$$ C_n^{[k]}  (G) := C_n^k (G)/C_n^{k-1} (G). $$
Passing to the quotient eliminates every boundary face whose degree is strictly less than $k$. Thus the induced boundary
\[
\partial^{[k]}:C_n^{[k]}(G)\longrightarrow C_{n-1}^{[k]}(G)
\]
is determined entirely by the degree-preserving faces of the cubical boundary. The homology groups $H_*^{[k]}(G)$ therefore isolate the contribution of singular cubes of degree exactly $k$.
We let
$$ H^{[k]}_*(G) := H_*(C^{[k]}_*(G)) $$
denote the homology of this quotient complex.

The cubical dimension of a graph $G$ is the dimension of the maximal
injective singular cube in $G$. If $G$ has cubical dimension $d$, then
\[
C_*^d(G)=C_*(G),
\]
and
\[
C_*^{[k]}(G)=0
\qquad\text{for }k>d.
\]
Hence the degree filtration is finite. 

For any graph $G$, every singular $n$-cube has degree at most $n$, therefore we have
\[
C_n^k(G)=C_n(G), \qquad k\ge n.
\]
Hence the degree filtration is exhaustive and bounded below.  

Recall (see, for example, \cite{weibel1994introduction}) that a filtered chain complex
$$ F_0C_* \le F_1 C_* \le F_2 C_* \le \cdots \le C_* = \bigcup_p F_pC_* $$
has an associated spectral sequence with $E^1$-page
$$ E^1_{p,q} = H_{p+q}\left( \frac{F_p C_*}{F_{p-1}C_*}\right) $$
and $d_1$-differentials
$$ d_1: E^1_{p,q} \to E^1_{p-1,q}. $$
Inductively, the $E^r$-page is a bigraded abelian group $E^r_{p,q}$ with differentials
$$ d_r: E^r_{p,q} \to E^r_{p-r,q+r-1} $$
and 
$$ E^{r+1}_{p,q} := H_{p,q}(E^r_{\bullet, *}, d_r). $$
In the case where the spectral sequence is first quadrant (i.e. $E^r_{p,q} = 0$ if $p$ or $q$ is negative) we have for each fixed $(p,q)$
$$ E^r_{p,q} \cong E^{r+1}_{p,q} $$
for $r \gg 0$, and $E^\infty_{p,q}$ is defined to be this stabilized group. The spectral sequence converges to $H_*(C_*)$ in the sense that there is a filtration
$$ F_0H_*(C_*) \le F_1H_*(C_*) \le F_2H_*(C_*) \le \cdots \le H_*(C_*) = \bigcup_p F_pH_*(C_*) $$
with
$$ \frac{F_pH_{p+q}(C_*)}{F_{p-1}H_{p+q}(C_*)} \cong E^\infty_{p,q}. $$
Thus the spectral sequence of the filtered chain complex interpolates between the homology of the associated graded
$$ H_*\left( \frac{F_\bullet C_*}{F_{\bullet-1}C_*}\right)$$
and the associated graded of the homology
$$ \frac{F_\bullet H_*(C_*)}{F_{\bullet-1}H_*(C_*)}. $$

\begin{defn}
The \emph{degree spectral sequence} is the first quadrant spectral sequence 
\[
E^1_{p,q}=H^{[p]}_{p+q}(G)\Longrightarrow H_{p+q}(G),
\]
associated to the filtered chain complex $\{C^k_*(G)\}_k$. 
\end{defn}

 The $E^1$-page of the  degree spectral sequence is shown in Figure \ref{fig_spseq}, where we write $H_{p+q}^{[p]}$  for $H_{p+q}^{[p]}  (G)$.

\begin{figure}[H]
  \centering
\begin{tikzpicture}
  \matrix (m) [matrix of math nodes,
    nodes in empty cells,nodes={minimum width=5ex,
    minimum height=5ex,outer sep=-5pt},
    column sep=1ex,row sep=1ex]{
    p~~     &  ~    & ~    &~     & ~   & ~   &\textcolor{white}{H_5^{[5]}} \\
          4     &  0 &  0  & 0 & 0 & H_4^{[4]} & \textcolor{white}{H_5^{[4]}}\\
          3     &  0 &  0  & 0 & H_3^{[3]} & H_4^{[3]} &\textcolor{white}{H_5^{[3]}} \\
          2     &  0 &  0  & H_2^{[2]} & H_3^{[2]} & H_4^{[2]} &\textcolor{white}{H_5^{[2]}} \\
          1     &  0 &  H_1^{[1]}  & H_2^{[1]} & H_3^{[1]} & H_4^{[1]} &\\
          0     &   H_0^{[0]} & 0 &  0  & 0 & 0 &\\
    \quad\strut &   0  &  1  &  2  & 3 & 4 & p+q &\strut \\};
    \draw[-stealth] (m-1-7.south west) -- node[pos=0.5, above left, xshift=3pt, yshift=-4pt] {$\scriptscriptstyle d_1$} (m-2-6.north east);
    \draw[-stealth] (m-2-7.south west) -- node[pos=0.5, above left, xshift=3pt, yshift=-4pt] {$\scriptscriptstyle d_1$} (m-3-6.north east);
    \draw[-stealth] (m-3-7.south west) -- node[pos=0.5, above left, xshift=3pt, yshift=-4pt] {$\scriptscriptstyle d_1$} (m-4-6.north east);
    \draw[-stealth] (m-4-7.south west) -- node[pos=0.5, above left, xshift=3pt, yshift=-4pt] {$\scriptscriptstyle d_1$} (m-5-6.north east);
    \draw[-stealth] (m-2-6.south west) -- node[pos=0.5, above left, xshift=3pt, yshift=-4pt] {$\scriptscriptstyle d_1$} (m-3-5.north east);
    \draw[-stealth] (m-3-5.south west) -- node[pos=0.5, above left, xshift=3pt, yshift=-4pt] {$\scriptscriptstyle d_1$} (m-4-4.north east);
    \draw[-stealth]  (m-4-4.south west) -- node[pos=0.5, above left, xshift=3pt, yshift=-4pt] {$\scriptscriptstyle d_1$} (m-5-3.north east);
    \draw[-stealth] (m-5-3.south west) -- node[pos=0.5, above left, xshift=3pt, yshift=-4pt] {$\scriptscriptstyle d_1$} (m-6-2.north east);
    \draw[-stealth] (m-3-6.south west) -- node[pos=0.5, above left, xshift=3pt, yshift=-4pt] {$\scriptscriptstyle d_1$} (m-4-5.north east);
    \draw[-stealth] (m-4-5.south west) -- node[pos=0.5, above left, xshift=3pt, yshift=-4pt] {$\scriptscriptstyle d_1$} (m-5-4.north east);
    \draw[-stealth] (m-4-6.south west) -- node[pos=0.5, above left, xshift=3pt, yshift=-4pt] {$\scriptscriptstyle d_1$} (m-5-5.north east);
\draw[thick] (m-1-1.east) -- (m-7-1.east) ;
\draw[thick] (m-7-1.north) -- (m-7-8.north west) ;
\end{tikzpicture}
  \caption{$E^1$-page of the degree spectral sequence}\label{fig_spseq}
\end{figure}
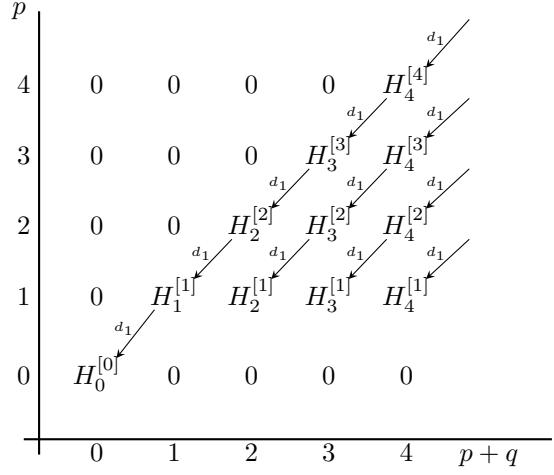

The $E^1$-page consists of the homology of the individual degree layers. The successive differentials describe how these layers interact, and the spectral sequence assembles this information to recover the discrete cubical homology of $G$. 

The edge of the $E^1$-page is the chain complex
\[
H^{[0]}_0(G)
\xleftarrow{d_1}
H^{[1]}_1(G)
\xleftarrow{d_1}
H^{[2]}_2(G)
\xleftarrow{d_1}
\cdots.
\]
Its homology forms the edge of the $E^2$-page. 

\begin{defn}\label{defn:Hinj}
Define the
\emph{injective homology} of $G$ to be this edge of the $E^2$-page:
\[
H_n^{\mathrm{inj}}(G) := H_n(H^{[\bullet]}_\bullet(G), d_1).
\]
\end{defn}

The term
\[
E^\infty_{n,0}(G)
\]
is the highest potentially non-trivial associated graded quotient of the filtration induced on
$H_n(G)$. Consequently, there are natural maps
\[
H_n(G)\twoheadrightarrow E^\infty_{n,0}(G)
\hookrightarrow H_n^{\mathrm{inj}}(G).
\]

\section{Degree-1 homology cycles in a triangle-free graph}\label{sec_H_n_1_equals_0}

In this section, our main goal is to prove that $H_n^{[1]}$ of a triangle-free graph is trivial for $n\geq 2$ (Theorem \ref{thm_H_2_1=0}). 
Since a triangle- and square-free graph has cubical dimension at most 1, the degree spectral sequence implies that 
this result extends the main result of \cite{barcelo2021SIAM}, which asserts that the discrete cubical homology of triangle- and square-free graphs is trivial in dimension $2$ and higher. 

To provide context and to clarify how the arguments of \cite{barcelo2021SIAM} carry over to our setting, we first give a streamlined approach to the main result  of~\cite{barcelo2021SIAM}.
The technique these authors used to prove Theorem~\ref{thm_SIAMpaper} is based on a combinatorial analog of the notion of covering space in the category of graphs.\footnote{This combinatorial notion of covering differs from the homotopical covering graph of \cite{kapulkin2025fundamental}, in that $\td{G}$ is always a forest.}

Given graph $G$, we will let $G_\RR$ denote the usual $1$-dimensional CW complex (topological graph) whose $0$-cells are given by the vertices of $G$, and $1$-cells are given by the edges of $G$.  Let
$$ p: \td{G}_\RR \to G_\RR $$
denote the universal cover.  The CW structure on $G_\RR$ lifts to a CW structure on $\td{G}_\RR$, giving $\td{G}_\RR$ the structure of a topological forest (a disjoint union of topological trees).  Let $\td{G}$ be the combinatorial forest associated to this topological forest.
Then the local homeomorphism $p$ arises from a local isomorphism of graphs    $$ \td{G} \to G. $$
If $G$ is path connected, $\td{G}$ is a tree, and the fundamental group
$$ \pi := \pi_1(G_\RR,v) $$
acts freely on $\td{G}$ through deck transformations. 
The following theorem is proven in \cite{barcelo2021SIAM}.

\begin{theorem}[{\cite{barcelo2021SIAM}}]\label{thm_SIAMpaper}
If $G$ is triangle-free and square-free, there is an isomorphism
$$ H_*(G) \cong H_*(G_\RR) $$
between the discrete cubical homology of $G$ and the singular homology of $G_\RR$.  In particular, 
$$ H_{\ge 2}(G) = 0. $$
\end{theorem}

\begin{proof}
We present a different proof which shares the philosophy of \cite{barcelo2021SIAM} in its use of the covering graph $\td{G}$. Without loss of generality, assume $G$ is path connected.
Let
$$ \sigma : I^n \to G $$
denote a singular $n$-cube.
By \cite[Lem.~4.3(3-4)]{barcelo2021SIAM}, since $G$ is triangle-free and square-free, a lift
$$ \td{\sigma} : I^n \to \td{G} $$
is uniquely determined by a lift of its base vertex.  Thus the group $\pi$ acts freely and transitively on the collection of lifts, and there is an isomorphism
\begin{equation}\label{eq:piequiv}
C_* (G) \cong C_*(\td{G}) \otimes_{\ZZ[\pi]} \ZZ.
\end{equation}
By our remarks above, $C_*(\td{G})$ is free as a $\ZZ[\pi]$-module.  Since $\td{G}$ is a tree, it is acyclic, and we therefore have
$$ H_*(C_*(\td{G})) = H_*(\td{G}) \cong 
\begin{cases}
\ZZ, & * = 0, \\
0, & * > 0.
\end{cases}
$$

This is because $\td{G}$
is a union of finite trees $\td{G}_i$, each of the $\td{G}_i$ is discretely contractible, and 
$$ H_*(\td{G}) \cong \varinjlim_i H_*(\td{G}_i). $$
It follows that $C_*(\td{G})$ is a free $\ZZ[\pi]$-resolution of $\ZZ$, and hence by (\ref{eq:piequiv})
$$ H_*(G) \cong H_*(\pi; \ZZ) \cong H_*(G_\RR). $$
\end{proof}

In Theorem~\ref{thm_SIAMpaper}, the  graph $G$ is required to be triangle- and square-free because this is the condition required to guarantee the existence of lifts of all singular $n$-cubes.  Topologically, the obstruction to lifting a singular $n$-cube
$$ \sigma: I^n \to G $$
is the non-triviality of the image of the induced map
$$ \sigma_*: \pi_1(I^n_\RR) \to \pi_1(G_\RR). $$

\begin{lemma}\label{lem:pi1}
The fundamental group $\pi_1(I^n_\RR)$ is generated by the images
$$ \pi_1(I^2_\RR) \to \pi_1(I^n_\RR) $$
arising from the $2$-dimensional faces
$$ f: I^2 \hookrightarrow I^n $$
of $I^n$.
\end{lemma}

\begin{proof}
To alleviate any ambiguities arising from basepoints, choose a maximal tree $T \subseteq I^n$, and let $I^n_\RR/T_\RR$ be the quotient with basepoint $\ast$.  Since $I^n_\RR \to I^n_\RR/T_\RR$ is a homotopy equivalence, it suffices to show that $\pi_1(I^n_\RR/T_\RR)$ is generated by the images of the composites
$$ \pi_1(I^2_\RR,0) \xrightarrow{f_*} \pi_1(I^n_\RR, f(0)) \to \pi_1(I^n_\RR/T_\RR, \ast). $$
Let $\Box^n$ be the solid topological $n$-cube (contractible).  The Van Kampen theorem implies that
since $I^n_\RR$ is the $1$-skeleton of $\Box^n$, and the $2$-cells of $\Box^n$ are its $2$-dimensional faces, we have 
$$ 0 = \pi_1(\Box^n/T_\RR, \ast) = \pi_1(I^n_\RR/T_\RR,\ast)/\bra{\Img f_* \: : \: \text{$f$ a $2$-dimensional face}}. $$
The result follows.
\end{proof}

We deduce the following (compare with \cite[Lem.~4.3(3-4)]{barcelo2021SIAM}).

\begin{theorem}\label{thm_LemmaSIAMpaper_modified}
    For a graph $G$, let 
    $$\sigma \colon I^n \to G$$ 
    be a singular $n$-cube. Suppose that for every 2-dimensional face $\tau$ of $\sigma$, $\Img (\tau)$ is not a $3$- or $4$-cycle. Then the lifting problem shown below as the solid diagram has a unique solution:     
    \begin{center}
    \begin{tikzcd}
    I^0 \arrow[r, "\widetilde{q}"] \arrow[d, "0"'] & \td{G} \arrow[d] \\
    I^n \arrow[r, "\sigma"]\arrow[ur, dashed, "\exists!~\td{\sigma}"] & G.
    \end{tikzcd}    
    \end{center}
\end{theorem}

\begin{proof}
The hypothesis implies that for any $2$-dimensional face
$$ I^2 \hookrightarrow I^n $$
the induced composite
$$ \pi_1(I^2_\RR) \to \pi_1(I^n_\RR) \to \pi_1(G_\RR) $$
is trivial.  By Lemma~\ref{lem:pi1}, it follows that
$$ \sigma_*: \pi_1(I^n_\RR) \to \pi_1(G_\RR) $$
is trivial.  Thus covering space theory implies 
$$ \sigma_\RR : I^n_\RR \to G_\RR $$
lifts uniquely to 
a map
$$ \td{\sigma}_\RR: I^n_\RR \to \td{G}_\RR $$ 

with $\td{\sigma}_\RR(0) = \td{q}$.  Because $\td{\sigma}_\RR$ is a lift of the topological graph map $\sigma_\RR$, and because the topological graph structure on $\td{G}_\RR$ is lifted from $G_\RR$, we deduce that there exists a unique graph map
$$ \td{\sigma}: I^n \to \td{G} $$
which underlies $\td{\sigma}_\RR$.
\end{proof}

\begin{corollary}\label{cor:lift}
Suppose that $G$ is triangle-free.  Then lifts of a degree-$1$ singular $n$-cube
$$ \sigma: I^n \to G $$
to $\td{G}$ are determined by lifts of $\sigma(0)$.
\end{corollary}

We deduce the following generalization of Theorem~\ref{thm_SIAMpaper}\footnote{Note that if $G$ is square-free, then $C^1_*(G) = C_*(G)$.}

\begin{prop}\label{prop:C_*^1}
If $G$ is triangle-free, then there is an isomorphism
$$ H_*(C_*^1(G)) \cong H_*(G_\RR). $$
\end{prop}

\begin{proof}
Since
$$ C^1_*\left( \coprod_i G_i \right) \cong \bigoplus_i C_*^1(G_i) $$
we may assume without loss of generality that $G$ is path connected.  Let
$$ \pi := \pi_1(G_\RR). $$
By Corollary~\ref{cor:lift}, $C^1_*(\td{G})$ is a chain complex of free $\ZZ[\pi]$-modules, and we have 
$$ C^1_*(G) \cong C^1_*(\td{G}) \otimes_{\ZZ[\pi]} \ZZ. $$
Furthermore, since $\td{G}$ is a tree, we have
$$ C_*^1(\td{G}) = C_*(\td{G}) $$
and, like in the proof of Theorem~\ref{thm_SIAMpaper}, we deduce
$$ H_*(C^1_*(\td{G})) = H_*(\td{G}) = \begin{cases}
\ZZ, & * = 0, \\
0, & * > 0.
\end{cases}
$$
It follows that $C^1_*(\td{G})$ is a free $\ZZ[\pi]$-resolution of $\ZZ$, and we have
$$ H_*(C^1_*(G)) \cong H_*(\pi; \ZZ) \cong H_*(G_\RR). $$
\end{proof}

The main result of this section is the following:

\begin{theorem}\label{thm_H_2_1=0}
Let $G$ be a triangle-free graph. Then we have
$$ H^{[1]}_{\ge 2}(G) = 0. $$
\end{theorem}

\begin{proof}
Consider the long exact sequence
$$ \cdots H_n^{[0]}(G) \to H_n(G_\RR) \to H_n^{[1]}(G) \to H_{n-1}^{[0]}(G) \cdots $$
associated to the short exact sequence
$$ 0 \to C_*^0(G) \to C_*^1(G) \to C^{[1]}_*(G) \to 0 $$
and Proposition~\ref{prop:C_*^1}. The result follows from the fact that $H_{\ge 2}(G_\RR) = 0$ and $H^{[0]}_{\ge 1}(G) = 0$.
\end{proof}

The degree spectral sequence immediately gives the following corollary.

\begin{corollary}\label{cor_2_homology_no_3_cycles}
For a triangle-free graph $G$, we have
    \[H_2 (G)\cong\frac{\ker ( H_2^{[2]}  (G)\to H_1^{[1]}  (G))}{\Img(H_3^{[3]}  (G)\to H_2^{[2]}  (G))} = H^{\mathit{inj}}_2(G).\]
Furthermore, for a triangle-free graph $G$ of cubical dimension $2$, $H_{\geq 3}(G)\cong H_{\geq 3}^{[2]}(G)$. 
\end{corollary}

\section{$H_n^{[n]}$ of $n$-quasimonophobic graphs}\label{sec_injective_cycles}
In this section, we show that for a certain class of graphs, which we call
\emph{$n$-quasimonophobic graphs}, the homology group $H_n^{[n]}(G)$ admits a
particularly simple description. We begin by introducing the notion of
$n$-quasimonophobic graphs.

The concept of \emph{monophobic cube neighborhoods} was introduced by Greene and the first author
in  \cite{monophobic} as a sufficient condition for nontriviality of certain elements in the discrete homotopy
groups and in the discrete cubical homology groups of a graph.  
In the present section, we give the definition of $n$-quasimonophobic graphs
and develop several results building on those in \cite{monophobic},
to facilitate
the computation of $H_n^{[n]}$ for $n$-quasimonophobic graphs.

An \textit{$n$-cube} in a graph is a subgraph isomorphic to $I^n$, whereas a \textit{rigid $n$-cube} in a graph is an \emph{induced} subgraph isomorphic to
$I^n$. A singular $n$-cube $\sigma$ is said
to be \textit{supported by} a [rigid] $n$-cube if the image of $\sigma$ is equal
to that cube. 

A rigid $n$-cube $Q$ is \textit{monophobic} in a graph $G$ if,
whenever a singular $(n+1)$-cube contains a face supported by $Q$, it also
contains at least one additional face supported by $Q$.  This relates to the framework of \cite{monophobic} in the sense that a rigid $n$-cube $Q\subseteq G$ is monophobic if and only if it possesses a monophobic neighborhood.

We shall say that a rigid $n$-cube $Q$ is \emph{quasimonophobic} in a graph $G$ if,
whenever a \underline{noninjective} singular $(n+1)$-cube contains a face supported by $Q$, it also
contains at least one additional face supported by $Q$.

\begin{defn}[$n$-monophobic graph]
\label{defn_monophobic}
Define a graph $G$ to be \emph{$n$-monophobic} if every $n$-cube $Q$ in $G$ is rigid and monophobic.
\end{defn}  

\begin{defn}[$n$-quasimonophobic graph]\label{defn_quasimonophobic}
Define a graph $G$ to be \emph{$n$-quasimonophobic} if every $n$-cube $Q$ in $G$ is rigid and quasimonophobic.  We shall say it is \emph{quasimonophobic} if it is $n$-quasimonophobic for every $n$.
\end{defn}  

Note the following:
\begin{itemize}
\item If a graph is $n$-monophobic, then it is $n$-quasimonophobic. 
\item Every graph is $0$-quasimonophobic, but the only graphs which are $0$-monophobic are those without edges.
\item A graph is $1$-monophobic if and only if it is triangle-free and square-free, and it is $1$-quasimonophobic if and only if it is triangle-free.
\item If $G$ has cubical dimension $d$, it is tautologically $(>d)$-monophobic (and therefore $(>d)$-quasimonophobic).
\end{itemize}

We give a few basic examples of $n$-(quasi)monophobic graphs from
\cite{monophobic}. The cycle graph $C_4$ is $2$-(quasi)monophobic,
and, more generally, the discrete $n$-cube $I^n$ is
$n$-(quasi)monophobic. On the other hand, the complete bipartite
graph $K_{2,3}$ is not $2$-(quasi)monophobic; see
Example~\ref{eg_K23}. More examples and details can be found in
\cite{monophobic}.

In this section, we focus on the degree-$n$ quotient complex $C^{[n]}_*(G)$ associated with an $n$-quasimonophobic graph $G$ and are primarily concerned with its homology in dimension $n$:
$$ H_n^{[n]}(G) = H_n(C_*^{[n]}(G)). $$
All boundary computations involving the boundary operator
\[
\partial_{n+1}^{[n]}:C_{n+1}^{[n]}(G)\longrightarrow C_n^{[n]}(G)
\]
are understood to be taken modulo singular cubes of degree less than $n$.
Since degree-$n$ singular $n$-cubes are precisely the injective singular
$n$-cubes, only injective singular $n$-cubes survive in
$C_n^{[n]}(G)$ and hence can contribute to $H_n^{[n]}(G)$.

Now, we proceed to show that $H_n^{[n]}$ of $n$-quasimonophobic graphs is given by $C_n^{[n]}(G)$ modulo the equivalence relation defined in the following theorem.

\begin{theorem}\label{thm:homologous}
Let $Q$ be an $n$-cube in any graph $G$, and let $\sigma$ and $\gamma$ be singular $n$-cubes in $G$  with
$
\mathrm{Im}(\sigma) = \mathrm{Im}(\gamma) = Q.
$
Then $\sigma$ is homologous to $\gamma$ up to a sign in $C_n^{[n]}$.    
\end{theorem}

We will prove Theorem~\ref{thm:homologous} through a sequence of lemmas.  These lemmas were originally proven in \cite{PhD_thesis_Samira} in the context of graphs which are digital images; we include the proofs here for completeness.

Note that $\sigma$ and $\gamma$ are related by a finite sequence of compositions of cube automorphisms
$
T_i, \ T_{i,j} \colon  I^n \to I^n, \quad i,j \in \{1,\ldots,n\}, \ i<j,
$
where $T_i$ reflects the $i$-th coordinate via $t_i \mapsto 1 - t_i$, and $T_{i,j}$ permutes the $i$-th and $j$-th coordinates.
To prove the desired result, we show that $\sigma$, $\sigma \circ T_i$, and $\sigma \circ T_{i,j}$ are homologous to each other, up to a sign, where homology consists solely of singular $(n+1)$-cubes of degree $n$. 

The simplest instance of the relation involving a coordinate reflection
occurs for singular $1$-cubes.
Consider a graph containing the edge $\{x,y\}$. Let $xy$ denote the
singular $1$-cube in the graph defined by $0\mapsto x$ and $1\mapsto y$, and let
$yx$ denote its reversal. Then
$xy$ is homologous to $-yx$, because
 $-xy-yx$ is the boundary of the singular $2$-cube $\rho$
defined by 
\[
\rho(0,0)=\rho(1,0)=\rho(1,1)=x,
\qquad
\rho(0,1)=y.
\]
To generalize this phenomenon to singular cubes of arbitrary dimension,
the key observation is that $\rho$ can be defined by setting
$f_1^-\rho=xy$ and prescribing $f_1^+\rho$ as the degenerate
singular $1$-cube constant at $x$. This construction leads to the
following result.
\begin{lemma}[\cite{PhD_thesis_Samira}]\label{lem_flip}
For a singular $n$-cube $\sigma$ in 
$C_n^{[n]}(G)$,
the chain $\sigma+\sigma\circ T_i$ lies in $\partial^{[n]}_{n+1}\!\bigl(C_{n+1}^{[n]}(G)\bigr).
$
\end{lemma}
\begin{proof}
We will show that $\sigma + \sigma\circ T_i= (-1)^i \partial_{n+1}^{[n]} \rho$ in $C_n^{[n]}(G)$, where
 $\rho \in C_{n+1}^{[n]}(G)$ is the singular $(n+1)$-cube given by:
\[ \rho(t_1,\dots,t_{n+1}) = \begin{cases}
\sigma(t_1,\dots,t_{i-1},t_{i+1},t_{i+2},\dots,t_{n+1}) &\text{ if } t_i=0 \\
\sigma(t_1,\dots,t_{i-1},0,t_{i+2},\dots, t_{n+1}) &\text{ if } t_i=1.
\end{cases} \]
We must compute the various faces of $\rho$.

First we compute $f_i^-\rho$ and $f^+_i\rho$:
\begin{align*}
f_i^-\rho(t_1,\dots,t_n) &= \rho(t_1,\dots,t_{i-1},0,t_{i},\dots,t_n) = \sigma(t_1,\dots,t_{i-1},t_{i},\dots,t_{n}) \\
f^+_i\rho(t_1,\dots,t_n) &= \rho(t_1,\dots,t_{i-1},1,t_{i},\dots,t_n) = \sigma(t_1,\dots,t_{i-1},0,\dots,t_{n}) 
\end{align*}
Here we see that $f_i^-\rho = \sigma$, and $f^+_i\rho$ is degenerate since it does not depend on coordinate $t_i$.

Next we compute $f^-_{i+1}\rho$ and $f^+_{i+1}\rho$:
\begin{align*}
f^-_{i+1}\rho(t_1,\dots,t_n) &= \rho(t_1,\dots,t_{i},0,t_{i+1},\dots,t_n) = \sigma(t_1,\dots,t_{i-1},0,t_{i+1},\dots,t_{n}) \\
f^+_{i+1}\rho(t_1,\dots,t_n) &= \rho(t_1,\dots,t_{i},1,t_{i+1},\dots,t_n) \\
&=\begin{cases}
\sigma(t_1,\dots,t_{i-1},1,t_{i+1},t_{i+2},\dots,t_{n}) &\text{ if } t_i=0 \\
\sigma(t_1,\dots,t_{i-1},0,t_{i+1},\dots, t_{n}) &\text{ if } t_i=1
\end{cases} \\
&= \sigma \circ T_i(t_1,\dots,t_n)
\end{align*}
Thus $f^-_{i+1}\rho$ is degenerate, and $f^+_{i+1}\rho = \sigma\circ T_i$.

Now we compute $f^-_j\rho$ and $f^+_j\rho$ for $j\not\in \{i,i+1\}$. For simplicity in the notation below, we will assume that $j<i$:
\begin{align*}
f^-_j\rho(t_1,\dots,t_n) &= \rho(t_1,\dots,0,t_j,\dots,t_n) \\
&= \begin{cases}
\sigma(t_1,\dots,0,t_j,\dots,t_{i-1},t_{i+1},t_{i+2},\dots,t_{n}) &\text{ if } t_i=0 \\
\sigma(t_1,\dots,0,t_j,\dots,t_{i-1},0,t_{i+2},\dots, t_{n}) &\text{ if } t_i=1
\end{cases}
\end{align*}
From the two cases above, we see that $f^-_j\rho$ can take on at most $2^{n-1}+2^{n-2} < 2^n$ values, so $f^-_j\rho$ is not injective, and thus $f^-_j\rho = 0 \in C_n^{[n]}(G)$. A similar computation shows that $f^+_j\rho = 0 \in C_n^{[n]}(G)$.

Summing the faces of $\rho$ in $C_n^{[n]}(G)$ gives only the two nonzero terms $f_i^-\rho=\sigma$ and $f^+_{i+1}\rho=\sigma\circ T_i$, and we have
\[ \partial_{n+1}^{[n]} \rho = \sum_j (-1)^j(f^-_j\rho - f^+_j\rho) = (-1)^i\sigma - (-1)^{i+1}\sigma\circ T_i = (-1)^i(\sigma + \sigma\circ T_i) \]
as desired.
\end{proof}

In order to show that $\sigma$ is homologous to
$-\sigma\circ T_{i,j}$, we first show that $\sigma$ is homologous to
$\sigma\circ R_{i,j}$, where
\[
R_{i,j}=T_{i,j}\circ T_j.
\]
The following example illustrates this relation in dimension $2$,
and the construction used in the example will be generalized in the
next lemma.

Consider a graph containing the $4$-cycle
$
(x,y,z,w,x),
$ 
and a singular $2$-cube $\sigma$ in the graph with
\[
\sigma(0,0)=x,\qquad
\sigma(1,0)=y,\qquad
\sigma(1,1)=z,\qquad
\sigma(0,1)=w.
\]
In order to show that \(\sigma\) is homologous to
\(\sigma\circ R_{1,2}=\sigma\circ T_{1,2}\circ T_2\),
we consider the singular \(3\)-cube \(\rho\) shown below, together with its boundary
\(\partial_3^{[2]}\rho\), where
\[
f_3^-\rho=\sigma
\qquad\text{and}\qquad
f_3^+\rho=\sigma\circ R_{1,2}.
\]
\begin{align}
\partial_3^{[2]}\left(
\raisebox{-0.5\height}{\includegraphics[scale=0.35]{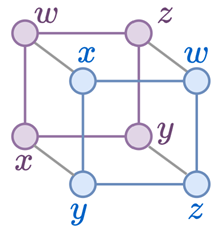}}
\right)
&=
\begin{alignedat}{2}
&-\raisebox{-0.4\height}{\includegraphics[scale=0.35]{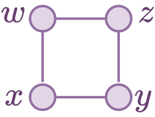}}
&&+
\raisebox{-0.4\height}{\includegraphics[scale=0.35]{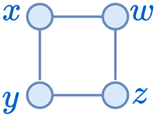}}
\\
&\quad\quad f_3^-\rho
&&\quad\quad f_3^+\rho .
\end{alignedat}
\label{eq_sing_3_cube_rotation}
\end{align}
All the other faces of $\rho$ are noninjective and therefore
have degree less than $2$, so they do not contribute to the boundary
in $C_*^{[2]}(G)$ and hence the displayed boundary shows that
$\sigma$ is homologous to $\sigma\circ R_{1,2}$ 
in \(H_2^{[2]}(G)\).

The construction of \(\rho\) can be generalized to arbitrary dimension, yielding the following result for a singular \(n\)-cube \(\sigma\).

\begin{lemma}
[\cite{PhD_thesis_Samira}]\label{lem_rotboundary}
For a singular $n$-cube $\sigma$ in 
$C_n^{[n]}(G)$,
the chain $\sigma-\sigma\circ R_{i,j}$ lies in $\partial^{[n]}_{n+1}\!\bigl(C_{n+1}^{[n]}(G)\bigr).
$
\end{lemma}
\begin{proof}
We show that $\sigma - \sigma \circ R_{i,j} = (-1)^{n+1}\,\partial_{n+1}^{[n]} \rho$ in $C_n^{[n]}(G)$, where $\rho \in C_{n+1}^{[n]}(G)$ is the singular $(n+1)$-cube whose front and back $(n+1)$-faces are $\sigma$ and $\sigma \circ R_{i,j}$, respectively. That is, 
$$\rho(t_1,\dots,t_{n+1}) =
\begin{cases}
\sigma(t_1,\dots,t_n), & \text{if } t_{n+1} = 0, \\[2pt]
\sigma \circ R_{i,j}(t_1,\dots,t_n), & \text{if } t_{n+1} = 1.
\end{cases}$$
Thus, it suffices to show that all remaining faces of $\rho$ are trivial in $C_n^{[n]}(G)$. 
Observe that for  $1 \le k \le n$, $f^-_k \rho$ has
$f^-_k \sigma$ and $f^-_k(\sigma \circ R_{i,j})$ as a pair of opposite $(n+1)$-faces.
The images of these two singular cubes have a nonempty intersection, because,
$f^-_k \sigma(I^{n-1})$ is a face of the elementary cube $\sigma(I^n)$, while
$f^-_k(\sigma \circ R_{i,j})(I^{n-1})$ is the same face of a quarter-turn rotation of that cube.
Consequently, $f^-_k \rho$ is not injective and is therefore trivial in $C_n^{[n]}(G)$. The argument for $f^+_k \rho$ is analogous.
\end{proof}
Next, we combine the above result with Lemma~\ref{lem_flip} to obtain
the desired relation between $\sigma$ and $\sigma\circ T_{i,j}$.
\begin{lemma}[\cite{PhD_thesis_Samira}]\label{lem_swapboundary}
For a singular $n$-cube $\sigma$ in 
$C_n^{[n]}(G)$,
the chain $\sigma+\sigma\circ T_{i,j}$ lies in $\partial^{[n]}_{n+1}\!\bigl(C_{n+1}^{[n]}(G)\bigr).
$
\end{lemma}
\begin{proof}
Since $T_{i,j} = R_{i,j} \circ T_j$, we have: 
\[ \sigma + \sigma\circ T_{i,j} = (\sigma - \sigma\circ R_{i,j}) + (\sigma\circ R_{i,j} + (\sigma\circ R_{i,j}) \circ T_j), \]
and the above is a chain in $\partial^{[n]}_{n+1}\!\bigl(C_{n+1}^{[n]}(G)\bigr)$ by Lemmas  \ref{lem_flip} and \ref{lem_rotboundary}.
\end{proof}

These lemmas conclude the proof of Theorem~\ref{thm:homologous}.

\begin{defn}
Let $\sim$ denote the equivalence relation on $C^{[n]}_n(G)$ where for injective singular $n$-cubes $\sigma$ and $\gamma$ with
$\Img(\sigma)=\Img(\gamma)$, we declare
\[
\sigma \sim (-1)^{[\sigma,\gamma]}\,\gamma.
\]
Here, $[\sigma,\gamma]$ denotes the minimal number of factors of the form
$T_i$ and $T_{i,j}$ appearing in any composition that expresses $\sigma$ in
terms of $\gamma$.
\end{defn}

The importance of this equivalence relation is given by the following lemma.

\begin{lemma}\label{lem:celliso}
There is an isomorphism 
$$ C^{[n]}_n/\!\sim ~\cong C^\mathit{cell}_n(CW^\Box(G)) $$
\end{lemma}

\begin{proof}
Let $\{D_i\}_i$ denote the set of subgraphs
$$ D_i \subseteq G $$
such that $D_i$ is isomorphic to $I^n$.  
Observe that we have 
$$ C^{[n]}_n(G) \cong \ZZ\{ \sigma: I^n \to G \: : \: \sigma \: \mr{injective.} \}. $$
Choosing a representative $\sigma_i$ for each equivalence class of $\sigma$ with
$$ \mr{Im}(\sigma_i) = D_i $$
we have 
$$ C_n^{[n]}/\!\sim ~\cong \ZZ\{[\sigma_i]\}_i. $$
By definition the $n$-cells of $CW^\Box(G)$ are in bijective correspondence with the subgraphs $D_i \subseteq G$.  Furthermore, the choice of representative $\sigma_i$ of $[\sigma_i]$ gives rise to a map
$$ \sigma_i: \Box^n \hookrightarrow CW^\Box(G) $$
which we will take to be the characteristic map of the cell associated to $D_i$.
By choosing a fixed orientation of $\Box^n$, the characteristic map $\sigma_i$ fixes a choice of orientation of the cell associated to $D_i$.  The result follows from the fact that the equivalence relation changes the sign of $[\sigma_i]$ when you change the orientation.
\end{proof}

We are now ready to state one of the main results of this section, which provides a
convenient description of $H_n^{[n]}(G)$ for $n$-quasimonophobic graphs.

\begin{theorem}\label{thm:cell}
For an $n$-quasimonophobic graph $G$, we have
\[
H_n^{[n]}(G)=C_n^{[n]}(G)/\!\sim ~\cong C_n^{\emph{cell}}(CW^\Box(G)).
\]
\end{theorem}

The proof of Theorem~\ref{thm:cell} is somewhat involved.  We pause to explain the strategy of the argument.  For any graph $G$, $C_{n-1}^{[n]}(G)$ is trivial, therefore, the $n$-th homology group $H_n^{[n]}(G)$ of the quotient complex $C^{[n]}(G)$, is given by
    \[
    H_n^{[n]}(G)
\;=\;
\frac{C_n^{[n]}(G)}{\partial^{[n]}_{n+1}\!\left(C_{n+1}^{[n]}(G)\right)}
\;\cong\;
\frac{C_n^{[n]}
(G)
\big/\sim}{\partial_{n+1}^{[n]}\!\left(C_{n+1}^{[n]}(G)\right)\big/\sim}.
    \]
    To obtain the desired result, we will use monophobicity arguments, following
\cite{monophobic}, to show that for an $n$-quasimonophobic graph $G$, the group
\[
\partial_{n+1}^{[n]}\!\left(C_{n+1}^{[n]}(G)\right)\big/\sim
\]
is trivial.

The following examples illustrate the vanishing of
$\operatorname{Im}(\partial_{n+1}^{[n]})$ modulo the equivalence relation
$\sim$ for $n$-quasimonophobic graphs in the cases $n=1$ and $n=2$. 
\begin{eg}
Consider the $1$-quasimonophobic graph consisting of the two edges
$\{x,y\}$ and $\{x,w\}$. The boundary of a singular $2$-cube in the
graph is shown below:
\begin{align}
\partial_2^{[1]}\left(
\raisebox{-0.5\height}{\includegraphics[scale=0.35]{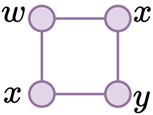}}
\right)
&=-xw+yx+xy-wx,
\label{eq:boundary_of_symm}
\end{align}
where a singular $1$-cube with $0\mapsto x$ and $1\mapsto y$ is
represented by $xy$, and similarly for the other terms. The singular
$1$-cubes $xw$ and $wx$ are related to each other by $T_1$, so
$xw\sim -wx$; similarly, $yx\sim -xy$. Thus the right-hand side of
\eqref{eq:boundary_of_symm} is trivial modulo $\sim$.
\end{eg}
\begin{eg}
Consider the following singular $3$-cube in the $4$-cycle
$
(x,y,z,w,x),
$
together with its boundary $\partial_3^{[2]}$. 

\begin{align}
\partial_3^{[2]}\left(
\raisebox{-0.5\height}{\includegraphics[scale=0.35]{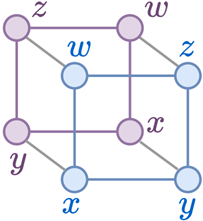}}
\right)&
=
\begin{alignedat}{4}
&-\raisebox{-0.4\height}{\includegraphics[scale=0.35]{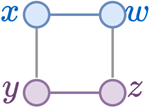}}
&&+\raisebox{-0.4\height}{\includegraphics[scale=0.35]{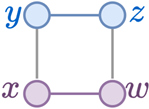}}
&&+\raisebox{-0.4\height}{\includegraphics[scale=0.35]{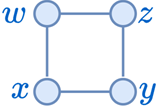}}
&&-\raisebox{-0.4\height}{\includegraphics[scale=0.35]{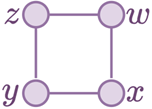}}\\
&\qquad\quad f_1^-
&&\qquad\quad  f_1^+
&&\qquad\quad  f_3^-
&&~~\qquad  f_3^+
\end{alignedat}
\label{eq:boundary_of_symm_3}
\end{align}
The $f_2^\pm$-faces (top and bottom faces) are both noninjective and therefore do not
contribute to the boundary above. Furthermore, the $f_1^-$- and $f_3^+$-faces (the right and the back faces) are related to each other by $T_{1,2}$, and hence are equivalent under $\sim$ up to a negative sign. The same is true for the left ($f_1^+$-) and front ($f_3^-$-) faces. 
Therefore, the
right-hand side is trivial modulo the equivalence relation $\sim$.
 \end{eg}

The preceding examples illustrate, in dimensions $2$ and $3$, the
general pattern that we now prove for arbitrary $n$. Let $\sigma$ be a
singular $(n+1)$-cube in $G$ of degree $n$. We analyze the boundary
$\partial_{n+1}^{[n]}(\sigma)$ in the quotient complex $C_*^{[n]}(G)$.

Since $\sigma$ has degree $n$, it is noninjective and has at least one
injective $n$-face, equivalently, at least one face supported by a
rigid $n$-cube in $G$. By $n$-quasimonophobicity, whenever a face of
$\sigma$ is supported by a rigid $n$-cube, at least one other face of
$\sigma$ is supported by the same rigid $n$-cube. The next lemmas
generalize the configurations illustrated in the preceding examples:
In Lemma~\ref{lem_mostly_noninjective_faces}, we show that all but at most four faces of $\sigma$ are
noninjective and hence vanish in $C_n^{[n]}(G)$
; In Lemma~\ref{lem_injective_face_relations}, we determine the
possible relations among the remaining injective faces and show that
they occur in pairs related by compositions of operators of the form
$T_i$ and $T_{i,j}$. These relations imply that
the remaining terms in
$\partial_{n+1}^{[n]}(\sigma)$ cancel modulo the equivalence relation
$\sim$. Thus
\[
\partial_{n+1}^{[n]}(\sigma)\equiv 0\pmod{\sim},
\]
and hence
\[
\operatorname{Im}\bigl(\partial_{n+1}^{[n]}\bigr)/\mathord{\sim}=0.
\]
This proves the theorem.

Before proving Lemmas~\ref{lem_injective_face_relations}~and~\ref{lem_mostly_noninjective_faces}, we establish two preliminary results.  The following result appears in \cite[Lemma 5.2]{jamil_digital}, and a similar result is also proved in \cite{monophobic}.

\begin{lemma}\label{lem_flipORrot}
Let $\sigma$ and $\gamma$ be injective singular $n$-cubes with images equal to the same rigid $n$-cube. Suppose that for all $t\in I^n$, $\sigma(t)$ is either equal or adjacent to $\gamma(t)$. Then $\sigma$ and $\gamma$ are either equal to each other, or related by a flip or a rotation operator. 
\end{lemma}

\begin{lemma}\label{lem_exactly_two possibilities}
    Let $\sigma$ be a noninjective nondegenerate singular $(n+1)$-cube with a face supported by an $n$-cube $Q$ in an $n$-quasimonophobic graph $G$. Then exactly one of the following holds for $\sigma$:
\begin{itemize}
    \item~\!\emph{[Case-1]} A pair of adjacent faces $f_i^a\sigma$ and $f_j^b\sigma$, with $i<j$ and $a,b\in\{-,+\}$, are supported by $Q$ such that
    $
        f_i^{-a}f_j^b\sigma = f_i^af_j^{-b}\sigma .
    $
    \item~\!\emph{[Case-2]}  A pair of opposite faces of $\sigma$ are supported by $Q$, with one face being a rotation $T_iT_{i,j}$ of the other.
\end{itemize}
\end{lemma}

\begin{proof}
Since $G$ is $n$-quasimonophobic and $\sigma$ is noninjective with a face supported by $Q$, there exists at least one additional face of $\sigma$ supported by $Q$. Any two such faces are either adjacent or opposite to each other.  We first show that if a pair of adjacent faces are supported by $Q$, then Case-1 holds. We then consider the case in which a pair of opposite faces is supported by $Q$ and show that this yields either Case-1 or Case-2.

\smallskip

\noindent
\textbf{Adjacent case.}
Suppose that a pair of adjacent faces $f_i^a\sigma$ and $f_j^b\sigma$, with $i<j$, are supported by $Q$. 
Since both $f_i^a\sigma$ and $f_j^b\sigma$ are supported by the rigid cube $Q$, the image of their common face
$$ f_i^{a}f_j^b\sigma
=
f_{j-1}^{b}f_i^a\sigma $$
is a face of $Q$.

Because $Q$ is rigid, an $(n-1)$-dimensional face of $Q$ determines its opposite $(n-1)$-dimensional face uniquely. Therefore, the faces of $f_i^a\sigma$ and $f_j^b\sigma$ opposite to their common $(n-1)$-dimensional face must coincide, \textit{i.e.}, using the face relation \eqref{eq_face relation},
\[
f_i^{-a}f_j^b\sigma
=
f_{j-1}^{-b}f_i^{a}\sigma
=
f_i^af_j^{-b}\sigma.
\]
See, for example, the singular $2$- and $3$-cubes whose boundary
calculations are shown in \eqref{eq:boundary_of_symm} and
\eqref{eq:boundary_of_symm_3}, respectively, where there are pairs of adjacent
faces mapping onto the same $1$- or $2$-cube, respectively.
\smallskip

\noindent
\textbf{Opposite case.}
Now suppose that a pair of opposite faces $f_i^a\sigma$ and $f_i^{-a}\sigma$ are supported by $Q$. These faces are injective singular $n$-cubes satisfying the hypothesis of Lemma~\ref{lem_flipORrot}. Since $\sigma$ is nondegenerate, $f_i^a\sigma$ and $f_i^{-a}\sigma$ are either related by a rotation or by a flip.

If they are related by a rotation, then Case-2 holds  (for example, see 
the singular $3$-cube in \eqref{eq_sing_3_cube_rotation}). 

Suppose that the faces $f_i^a\sigma$ and $f_i^{-a}\sigma$ are related by a flip $T_k$ where $k\in\{1,\ldots,n\}$ is an index, i.e.

\[
f_i^{-a}\sigma = f_i^a\sigma \circ T_{k}.
\]
The singular $3$-cube in \eqref{eq:boundary_of_symm_3} provides an
example of this case, where the $f_1^\pm$-faces are related by $T_2$ as well as $f_3^\pm$-faces are related by $T_1$. 

A direct computation shows that

\begin{equation}\label{eq_codim22}
\begin{array}{rclr}
  f_i^{-a} f_{k+1}^b \sigma & = & f_i^a f_{k+1}^{-b} \sigma & \text{if } i \le k, \\
  f_k^{b} f_i^{-a} \sigma   & = & f_k^{-b} f_i^{a} \sigma   & \text{if } i > k.
\end{array}    
\end{equation}

For example, assuming $i\leq k$: for any $t=(t_1,\ldots,t_{i-1},t_{i+1},\ldots,t_{k},t_{k+2}\ldots,t_{n+1})\in I^{n-1}$, setting $\delta_a=0$ if $a=-$ and $\delta_a=1$ if $a=+$, we get

\begin{align*}
&f_i^{-a} f_{k+1}^b\sigma(t) \\
&= \sigma(t_1,\ldots,t_{i-1},1-\delta_a,t_{i+1},\ldots,t_{k},\delta_b,t_{k+2},\ldots,t_{n+1}) \\
&= f_i^{-a}\sigma(t_1,\ldots,t_{i-1},t_{i+1},\ldots,t_{k},\delta_b,t_{k+2},\ldots,t_{n+1}) \\
&= f_i^a\sigma\!\circ T_k(t_1,\ldots,t_{i-1},t_{i+1},\ldots,t_{k},\delta_b,t_{k+2},\ldots,t_{n+1}) \\
&= f_i^a\sigma(t_1,\ldots,t_{i-1},t_{i+1},\ldots,t_{k},1-\delta_b,t_{k+2},\ldots,t_{n+1}) \\
&=\sigma(t_1,\ldots,t_{i-1},\delta_a,t_{i+1},\ldots,t_{k},1-\delta_b,t_{k+2},\ldots,t_{n+1}) \\
&= f_i^af_{k+1}^{-b}\sigma(t).
\end{align*}  

When $i\leq k$, $f_i^a\sigma$ and $f_{k+1}^b\sigma$ have a common face and \eqref{eq_codim22} implies that
$$ f_i^{-a}f^b_{k+1}\sigma = f_i^af_{k+1}^{-b}\sigma, $$
showing that the faces of $f_i^a\sigma$ and $f_{k+1}^b\sigma$ opposite to their common face are equal. This means that $f_i^a\sigma$ and $f_{k+1}^b\sigma$ are both supported by the same cube $Q$. Thus, Case-1 holds in this situation, with $j=k+1$.

 Similarly, if $i>k$,
 $f_i^a\sigma$ and $f_{k}^b\sigma$ are both supported by the same cube $Q$. Thus, Case-1 holds in this situation, with the role of  $(i,j)$ and $(a,b)$ played by $(k,i)$ and $(b,a)$, respectively.
\end{proof}

The preceding lemma shows that in case-2, faces supported by the same
quasimonophobic $n$-cube are related by a composition of the form $T_i T_{i,j}$.
A different relation holds in case-1, as described in the next lemma;
in that case, the faces are related by a composition of $j-i-1$ factors or $j-i$ factors. The lower-dimensional examples in \eqref{eq:boundary_of_symm} and
\eqref{eq:boundary_of_symm_3}
illustrate simple instances of Case-1(b) of the following lemma, where the relevant
faces are related by a single factor of the form $T_i$ or $T_{i,j}$. For example, if $\tau$ denotes the singular $3$-cube in \eqref{eq:boundary_of_symm_3}, then $f_1^\pm\tau=f_3^\mp\tau \circ T_{1,2}$.

\begin{lemma}\label{lem_injective_face_relations}
If $f_i^a\sigma$ and $f_j^b\sigma$ are a pair of adjacent faces of $\sigma$ with
$i<j$, both supported by the same rigid $n$-cube.
Then the following hold:
\begin{itemize}
    \item~\!\emph{[Case-1(a)]} If $a=b$, then
    \begin{align*}
f_i^a \sigma &= f_j^a \sigma\circ T_{i,i+1}\circ T_{i+1,i+2}\circ\cdots\circ T_{j-2,j-1}.
\end{align*}

    \item~\!\emph{[Case-1(b)]} If $a= -b$, then
    \begin{align*}
f_i^a \sigma &= f_j^{-a} \sigma\circ T_i\circ T_{i,i+1}\circ T_{i+1,i+2}\circ\cdots\circ T_{j-2,j-1}.
\end{align*}
\end{itemize}
\end{lemma}
\begin{proof}
When $a=b$, we first claim that $\sigma = \sigma \circ T_{i,j}$.
Indeed, the maps $\sigma$ and $\sigma \circ T_{i,j}$ agree on any point
$s=(s_1,\ldots,s_{n+1}) \in I^{n+1}$ with $s_i=s_j$.
They also agree when $s_i = 1 - s_j$, since 
$
f_i^{-a} f_j^{a} \sigma = f_i^{a} f_j^{-a} \sigma
$
by Lemma~\ref{lem_exactly_two possibilities}.

%
%
%
%
Similarly, for the case $a=-b$,
we can show that $\sigma\circ T_i = \sigma\circ T_i \circ T_{i,j}$, using the relation
$
f_i^{a} f_j^{a} \sigma = f_i^{-a} f_j^{-a} \sigma
$
of Lemma~\ref{lem_exactly_two possibilities}.

The proofs of the relations between faces are analogous in the cases $a=b$ and $a=- b$.
We therefore prove only the relation between $f_i^a\sigma$ and $f_j^{-a}\sigma$ in the case $a=-b$, omitting the proof for the other case.

Let $t=(t_1,\ldots,t_{i-1},t_{i+1}\ldots,t_{n+1}) \in I^{n}$ and set
$\delta_a=0$ if $a=-$
and $\delta_a=1$ if $a=+$.
\begin{align*}
    f_i^a\sigma(t)
&=\sigma(t_1,\ldots,t_{i-1},\delta_a,t_{i+1},\ldots,t_{n+1})\\
&=\sigma\circ T_i(t_1,\ldots,t_{i-1},1-\delta_a,t_{i+1},\ldots,t_{n+1})\\
&=\sigma\circ T_i\circ T_{i,j}(t_1,\ldots,t_{i-1},1-\delta_a,t_{i+1},\ldots,t_{n+1})\\
&=\sigma(t_1,\ldots,t_{i-1},1-t_j,t_{i+1},\ldots,t_{j-1},1-\delta_a,t_{j+1},\ldots,t_{n+1})\\
&=f_j^{-a}\sigma\circ T_i(t_1,\ldots,t_{i-1},t_j,t_{i+1},\ldots,t_{j-1},t_{j+1},\ldots,t_{n+1})\\
&=f_j^{-a}\sigma\circ T_i\circ T_{i,i+1}\circ T_{i+1,i+2}\circ\cdots\circ T_{j-2,j-1}(t).
\end{align*}
\begin{align*}
\end{align*}
\end{proof}
The following lemma shows that all the other faces are noninjective.

\begin{lemma}\label{lem_mostly_noninjective_faces}
    A nondegenerate degree-$n$ singular $(n+1)$-cube in an $n$-quasimonophobic graph has exactly $2$ or $4$ injective faces.
\end{lemma}

\begin{proof}
Let $\sigma$ be a nondegenerate degree-$n$ singular $(n+1)$-cube in an $n$-quasimonophobic graph. Note that $\sigma$ is noninjective, and has a face supported by an $n$-cube $Q$. Then by Lemma \ref{lem_exactly_two possibilities}, exactly one of the two cases holds.

\textbf{Case-1 of Lemma \ref{lem_exactly_two possibilities}}: We claim that the faces $f_k^c\sigma$, for all indices $k\not\in\{i,j\}$ and $c\in\{-,+\}$ are noninjective singular $n$-cubes, as illustrated by the $f_2^\pm$-faces of the singular 3-cube in \eqref{eq:boundary_of_symm_3}.

Note that since we have $f_i^a f_j^{-b}\sigma = f_i^{-a} f_j^{b}\sigma$, two codimension-2 faces of $f_k^c\sigma$ are equal, i.e.
\begin{align*}
  f_i^a f_j^{-b}(f_k^c\sigma) &= f_i^{-a} f_j^{b}(f_k^c\sigma)  &\text{ if }i<j<k,\\
  f_i^a f_{j-1}^{-b}(f_k^c\sigma) &= f_i^{-a} f_{j-1}^{b}(f_k^c\sigma)  &\text{ if }i<k<j,\\
  f_{i-1}^a f_{j-1}^{-b}(f_k^c\sigma) &= f_{i-1}^{-a} f_{j-1}^{b}(f_k^c\sigma)  &\text{ if }k<i<j.
\end{align*}
Note that the shift in indices in the second and third equations above comes from the face relations.  For example, for $i<k<j$, using face relations, we get 
\begin{align*}
    f_i^a f_{j-1}^{-b}(f_k^c\sigma) 
    &=f_i^a f_k^c f_{j}^{-b}\sigma 
     \\&=f_{k-1}^c(f_i^a  f_{j}^{-b}\sigma) 
    \\&=f_{k-1}^c(f_i^{-a}  f_{j}^{b}\sigma)
    \\&=f_i^{-a} f_{j-1}^{b}(f_k^c\sigma).  
\end{align*}

Thus if $i<j<k$, the image of $f_k^c\sigma$ agrees on points $s,t\in I^{n-1}$ such that $s_\ell=t_\ell$ for all $\ell\notin\{i,j\}$, and $(s_i,s_j)=(\delta_a,1-\delta_b)$ while $(t_i,t_j)=(1-\delta_a,\delta_b)$, and similarly, for other cases on the ordering of $i,j,k$.

The four other faces $f_k^a \sigma$, where $k \in \{i,j\}$ and $a \in \{+,-\}$, are either all injective, or else any noninjective face forces the face related to it by a bijection (by Lemma~\ref{lem_injective_face_relations}) to be noninjective as well.

\textbf{Case-2 of Lemma \ref{lem_exactly_two possibilities}.}
In this case, the two faces supported by $Q$ are rotations of each other.
We show that all the remaining faces of $\sigma$ are noninjective.

This is similar to the argument for the singular cube $\rho$ in
Lemma~\ref{lem_rotboundary}, also illustrated by the singular $3$-cube in
\eqref{eq_sing_3_cube_rotation}. There, the two $f_{n+1}^-$- and
$f_{n+1}^+$-faces are rotations of each other, while every other face is
noninjective. Indeed, each of these other faces has a pair of opposite
faces whose images have nonempty intersection, and hence is
noninjective. The same argument applies here, showing that all faces of
$\sigma$, except for the two faces related by rotation and supported by
$Q$, are noninjective. 
\end{proof}

\begin{proof}[\textbf{Completion of the proof of Theorem~\ref{thm:cell}}]
Now we show that the boundary $\partial_{n+1}^{[n]}(\sigma)$ is trivial modulo the
equivalence relation $\sim$ in all the three cases $1(a)$, $1(b)$, and $2$ of
Lemmas~\ref{lem_exactly_two possibilities} and~\ref{lem_injective_face_relations}.

\paragraph{Case $1(a)$.}
In this case, at least one or both of the two pairs of faces
$(f_i^-\sigma, f_j^-\sigma)$ and $(f_i^{+}\sigma, f_j^{+}\sigma)$ consists of
injective faces related to each other by the relation stated in Lemma \ref{lem_injective_face_relations}. In the boundary computation of
$\partial_{n+1}^{[n]}(\sigma)$, all the terms corresponding to injective
faces  $(f_i^a\sigma, f_j^a\sigma)$, for $a\in\{-,+\}$, cancel each other modulo the equivalence relation $\sim$: Set
$\delta_a = 0$ if $a=-$ and $\delta_a = 1$ if $a=+$. Then the term 
\begin{align*}
(-1)^{i+\delta_a} f_i^a\sigma
&=
(-1)^{i+\delta_a} f_j^a \sigma \circ T_{i,i+1}\circ T_{i+1,i+2}\circ\cdots\circ T_{j-2,j-1} \\
&\sim
(-1)^{i+\delta_a}(-1)^{j-i-1} f_j^a \sigma \\
&=
(-1)^{j+\delta_a-1} f_j^a \sigma,
\end{align*}
cancels with the term $(-1)^{j+\delta_a} f_j^a\sigma$.

The argument for
Case $1(b)$ is entirely similar.

\paragraph{Case $2$.}
In this case, the boundary  $\partial_{n+1}^{[n]}(\sigma)$ consists of the injective faces are $f_i^a\sigma$ and
$f_i^{-a}\sigma$, and since each is a rotation of the other, we have $f_i^a\sigma \sim f_i^{-a}\sigma$.
Consequently, the corresponding terms cancel modulo $\sim$, and the boundary is
trivial.
\end{proof}

Using Theorem~\ref{thm:cell},  we deduce the following.

\begin{theorem}\label{thm:iso}
Suppose that $G$ is $n$-quasimonophobic.  Then the map
$$ H_n(CW^\Box(G)) \to H_n^{\mathit{inj}}(G) $$
is an injection.  If in addition $G$ is $(n-1)$-quasimonophobic, then this map is an isomorphism.
\end{theorem}

\begin{proof}
Recall by Lemma~\ref{lem:celliso} that 
for any graph $G$, the cellular chains of $CW^\Box(G)$ in dimension $n$
is naturally identified with $C_n^{[n]}(G)/\!\sim$:
\[
C^{\emph{cell}}_n\bigl(CW^\Box(G)\bigr)
=\frac{C_n^{[n]}(G)}{\sim}.
\]
Moreover, it follows from Theorem~\ref{thm:homologous}, Lemma~\ref{lem_flip}, and Lemma~\ref{lem_swapboundary} that there is a surjection
$$ C_i^{[i]}(G)/\!\sim \twoheadrightarrow H^{[i]}_{i}(G). $$
Defining $K_i$ to be the kernel of this map, there is a short exact sequence of chain complexes given by
\begin{equation}\label{eq:SES}
 0 \to K_i \to C^\mathrm{cell}_i(CW^\Box(G)) \to H^{[i]}_i(G) \to 0.
 \end{equation}
The long exact sequence associated to the short exact sequence (\ref{eq:SES}) gives an exact sequence:
$$ H_n(K_\bullet) \to H_n(CW^\Box(G)) \rightarrow H_n^{\mathit{inj}}(G) \to H_{n-1}(K_\bullet). $$
If $G$ is $i$-quasimonophobic, Theorem~\ref{thm:cell} implies that 
$$ K_i = 0. $$
The result follows.
\end{proof}

\begin{theorem}\label{thm:main}
Let $G$ be a graph that is $(\le 2)$-quasimonophobic.
Then
\[
H_2(G)=H_2\bigl(CW^\Box(G)\bigr).
\]
\end{theorem}

\begin{proof}
Since $G$ is $1$- and $2$-quasimonophobic Theorem~\ref{thm:iso} implies that 
$$ H^\emph{inj}_2(G) \cong H_2(CW^\Box(G)). $$

Moreover, by Theorem~\ref{thm_H_2_1=0}, we have
$H_2^{[1]}(G)=0$ because $G$ is $1$-quasimonophobic (hence triangle-free).

Therefore, for $H_2(G)$, the only nontrivial contribution in the degree spectral sequence comes from
$H^\mathit{inj}_2(G)$, and this agrees with the cellular homology group.
Hence
\[
H_2(G)=H_2\bigl(CW^\Box(G)\bigr).
\]
\end{proof}

The following example shows that the assumption of $2$-quasimonophobicity
is necessary in Theorem~\ref{thm:main}.

\begin{eg}\label{eg_K23}
Consider the complete bipartite graph $K_{2,3}$ with bipartition
$\{v_1,v_2\}$ and $\{u_1,u_2,u_3\}$.
In the sense of discrete homotopy theory,
$K_{2,3}$ is contractible, since it deformation retracts onto one
of its edges (see~\cite{barcelo2001foundations} for the definition of deformation retracts).
Consequently,
\[
H_2(K_{2,3}) = 0 .
\]

The cubical CW--complex $CW^\Box(K_{2,3})$
obtained by filling in the cubes of $K_{2,3}$ is homeomorphic to
$S^2$.

The graph $K_{2,3}$ is not $2$-quasimonophobic.
Indeed, the non-injective singular $3$-cube
\[
\sigma \colon I^3 \to K_{2,3}
\]
given by
\begin{align*}
\sigma(0,0,0)&=v_1, &
\sigma(0,1,0)&=u_1, &
\sigma(1,0,0)&=u_2, &
\sigma(1,1,0)&=v_2, \\
\sigma(0,0,1)&=u_3, &
\sigma(1,1,1)&=u_3, &
\sigma(0,1,1)&=v_1, &
\sigma(1,0,1)&=v_1
\end{align*}
witnesses the nonmonophobicity of the
$2$-cubes in $K_{2,3}$.
Its boundary kills the injective $2$-cycle in $C^{[2]}_2(K_{2,3})/\sim$ corresponding to the generator of
$$ H_2(CW^\Box(K_{2,3})) = \ZZ, $$
showing that the map
$$ C^{[2]}_2(K_{2,3})/\sim \to H^{[2]}_2(K_{2,3}) $$
is not injective.

Thus, Theorem~\ref{thm:main} fails without the
assumption of $2$-quasimonophobicity.
\end{eg}

\section{Applications and  questions}\label{sec_[examples]}

In this section, we apply our main result (Theorem \ref{thm:main}) to compute the second discrete cubical homology of the Greene sphere $G_n^\mathrm{sph}$, for $n \ge 4$.  This computation was previously unknown for $n > 4$. We formulate a question concerning how such computations might generalize for other quasimonophobic graphs. We then explain how our results imply that quasimonophobicity can be used to geometrically produce non-trivial classes in discrete cubical homology.

\begin{defn}\label{defn_Greenes_sphere}
Define the $n$th Greene sphere $G^{\mathrm{sph}}_n$ as the graph obtained from a $2n$-cycle
(with vertices $0,\dots,2n-1$) by adjoining two vertices $u$ and $v$,
where $u$ is adjacent to all odd vertices and $v$ to all even vertices.
\end{defn}

\begin{theorem}\label{thm_Greenes_sphere}
    $H_2(G^{\mathrm{sph}}_n)\cong \mathbb{Z}$ and $H_3(G^{\mathrm{sph}}_n)=H^{[2]}_3(G^{\mathrm{sph}}_n)$, for all $n \ge 4$.
\end{theorem}
\begin{proof}
The CW complex $CW^\Box(G^\mr{sph}_n)$ is homeomorphic to $S^2$ (the graph $G^{\mr{sph}}_3 = I^3$, and hence is contractible).
Moreover, the graphs $G^{\mathrm{sph}}_n$ are triangle-free
($1$-quasimonophobic), have cubical dimension $2$, and are
$2$-monophobic \cite{monophobic}.  It follows that these graphs are quasimonophobic and thus the result follows.

Furthermore, since $G^{\mathrm{sph}}_n$ contains no $3$-cubes, 
$H_{3}^{[3]}(G^{\mathrm{sph}}_n)$ and 
$
H_{4}^{[3]}(G^{\mathrm{sph}}_n)
$
are both trivial,
and by Corollary \ref{cor_2_homology_no_3_cycles},
$
H_3(G^{\mathrm{sph}}_n)=H^{[2]}_3(G^{\mathrm{sph}}_n).
$
\end{proof}

Note that the computer computations of \cite{kapulkin2024efficient} actually establish that
$$ H_{3}(G^\mr{sph}_4) \cong H_{3}(CW^\Box(G^\mr{sph}_4)) = 0. $$
We also saw that $1$-quasimonophobicity eliminated the contributions of $H^{[1]}_{\ge 2}$ from the degree spectral sequence.  Motivated by these data points, it is natural to ask the following:

\begin{Quest}\label{Conj_H_2_CW_complex}
Suppose $G$ is a graph that is $(\le n)$-quasimonophobic. Does this imply that $H_n(G)$ is isomorphic to the homology group $H_n(CW^\Box(G))$?
\end{Quest}

Finally we explain how quasimonophobicity allows one to geometrically produce non-trivial classes in discrete cubical homology.

Suppose that 
$$ X_\bullet: \Box^{op}_{inj} \to \mr{Sets} $$ is a 
semi-cubical set, and let 
$$ G(X_\bullet) $$
denote the graph given by its 1-skeleton.
We will assume $X_\bullet$ satisfies the following conditions:
\begin{enumerate}
\item The graph $G(X_\bullet)$ is simple.
\item
$X_\bullet$ is \emph{regular} in the sense that its geometric realization $\abs{X}$ is a regular CW complex: each characteristic map
$$ \Box^n \to \abs{X_\bullet} $$
is injective.
\item For each $n$, the map
$$ X_n \to \{\text{$n$-cube subgraphs of $G(X_\bullet)$}\} $$
is an injection (i.e. every $n$-cube of $X_\bullet$ is determined by its $1$-skeleton).
\end{enumerate}
Suppose that
$$ x = \sum_i a_i \sigma_i \in C_n(X_\bullet) = \ZZ X_n $$  
is an $n$-cycle in the cubical $n$-chains of $X_\bullet$.

The class $x$ can be regarded as giving a cycle in the cellular chain complex 
$$ x \in C^\mr{cell}_n(\abs{X_\bullet}), $$
giving a class
$$ [x] \in H_n(\abs{X_\bullet}). $$
Conditions (2)-(3) above imply that there is an inclusion of CW-complexes
\begin{equation}\label{eq:CWinc}
 \abs{X_\bullet} \hookrightarrow CW^\Box (G(X_\bullet))
 \end{equation}
and we let  
$$ [x]_{CW} \in H_n(CW^\Box(G(X_\bullet))) $$
denote the image of $[x]$ under this map.
Alternatively, we may regard each $\sigma_i \in X_n$ as a map
$$ \sigma_i: \Box^n_\bullet \to X_\bullet $$
giving a map of graphs
$$ \sigma_i: I^n = G(\Box^n_\bullet) \to G(X_\bullet). $$
The resulting class
$$ x = \sum_i a_i \sigma_i \in C_n(G(X_\bullet)) $$
is a cycle in the discrete cubical chains in $G(X_\bullet)$, yielding a cycle
$$ [x]_\mr{graph} \in H_n(G(X_\bullet)). $$
Under the maps (\ref{eq:edge}) and (\ref{eq:cell})
$$ H_n(G(X_\bullet)) \to H_n^\mit{inj}(G(X_\bullet)) \leftarrow H_n(CW^\Box(G(X_\bullet))) $$
the classes $[x]_\mr{graph}$ and $[x]_{CW}$ map to the same element.

 Given a graph $G$ and a graph map
 $$ \alpha: G(X_\bullet) \to G $$
 the image of $[x]_\mathrm{graph}$
 under the map
 $\alpha_*: H_n(G(X_\bullet)) \to H_n(G)$ gives a ``geometric'' discrete cubical homology class
 $$ \alpha_*[x]_\mr{graph} \in H_n(G). $$
 If the map $\alpha$ is an injection, it induces an inclusion of CW complexes
 $$ CW^\Box(G(X_\bullet)) \hookrightarrow CW^\Box(G), $$
 and the image of $[x]_{CW}$ under this map gives a homology class
 $$ \alpha_*[x]_{CW} \in H_n(CW^\Box(G)). $$

 \begin{theorem}\label{thm:geom}
 Suppose that $G$ is $n$-quasimonophobic,
 $X_\bullet$, $x$ are as above,  
 $$ \alpha: G(X_\bullet) \hookrightarrow G $$
 is an inclusion of graphs, and the homology class
 $$ \alpha_*[x]_{CW} \in H_n(CW^\Box(G)) $$
 is non-trivial. Then
 the homology class
 $$ \alpha_*[x]_\mr{graph} \in H_n(G) $$
is non-trivial.
 \end{theorem}

\begin{proof}
Consider the diagram
$$
\xymatrix{
H_n(G(X_\bullet)) \ar[r] \ar[d]_{\alpha_*} & H^\mit{inj}_n(G(X_\bullet)) \ar[d]_{\alpha_*} & H_n(CW^\Box(G(X_\bullet))) \ar[l] \ar[d]^{\alpha_*} \\
H_n(G) \ar[r] & H_n^\mit{inj}(G)  & H_n(CW^\Box(G)) \ar@{_{(}->}[l]^{(*)}  
}
$$
induced by (\ref{eq:edge}) and (\ref{eq:cell}).  Note that $n$-quasimonophobicity implies the map $(*)$ in the above diagram is an injection by Theorem~\ref{thm:iso}.  The result follows.
\end{proof}

For example, for $n \ge 4$ there are clearly regular semi-cubical sets $S^2(n)_\bullet$ whose realization is $S^2$, we have 
$$ G^\mr{sph}_{n} = G(S^2(n)_\bullet), $$
and for $n \ge 4$ the cannonical map gives a homeomorphism
$$ S^2 = \abs{S^2(n)_\bullet} \xrightarrow{\approx} CW^\Box(G^\mr{sph}_n). $$
The graph $G^\mr{sph}_n$ is $2$-quasimonophonic, the fundamental class
$$ [S^2] \in H_2(S^2(n)_\bullet) $$
has non-trivial image
$$ 0 \ne [S^2]_{CW} \in H_2(CW^\Box(G^\mr{sph}_n)) \cong \ZZ $$
and hence gives the (nontrivial) generator of 
$$ [S^2]_\mr{graph} \in H_2(G^{sph}_n) \cong \ZZ. $$

However, in the case of $n = 3$, while $G^\mr{sph}_3 = I^3$ is still $2$-quasimonophobic, we have
$$ CW^\Box(G^\mr{sph}_3) = \Box^3 $$
and hence the image of $[S^2]$, 
$$ [S^2]_{CW} \in H_2(CW^\Box(G^\mr{sph}_3)) = 0, $$
is trivial.
The element
$$ [S^2]_\mr{graph} \in H_2(G^\mr{sph}_3) = 0 $$
is also trivial, since $G^\mr{sph}_3 = I^3$ is contractible. 

Finally, we saw
(Example~\ref{eg_K23}) that the graph $K_{2,3}$ is not $2$-quasimonophobic.  There is a semi-cubical set $\td{S}^2_\bullet$, obtained by gluing three $\Box^2_\bullet$'s together, with the property that:
\begin{itemize}
\item $\abs{\td{S}^2_\bullet} \approx S^2$,
\item $G(\td{S}^2_\bullet) = K_{2,3}$.
\end{itemize}
Since $CW^\Box(K_{2,3}) \approx S^2$, the fundamental class
$$ [S^2] \in H_2(\td{S}^2_\bullet) $$
gives a non-trivial class
$$ 0 \ne [S^2]_{CW} \in H_2(CW^\Box(K_{2,3})) = \ZZ. $$
However, since $K_{2,3}$ is contractible, the class
$$ [S^2]_\mr{graph} \in H_2(K_{2,3}) = 0 $$
is trivial.  This demonstrates the necessity of the quasimonophobicity condition in Theorem~\ref{thm:geom}. 
 
\bibliography{Ref}
\bibliographystyle{amsalpha} 
\end{document}